\documentclass[11pt,leqno]{article}
\usepackage[utf8]{inputenc}
\usepackage{a4wide}
\usepackage{amssymb,latexsym,amsmath,epsfig,pdfsync, amsthm,nicefrac,enumerate,bbm,color}

\newtheorem{theorem}{Theorem}[section]
\newtheorem{corollary}[theorem]{Corollary}
\newtheorem{lemma}[theorem]{Lemma}
\newtheorem{proposition}[theorem]{Proposition}

\newtheorem{remark}[theorem]{Remark}

\numberwithin{equation}{section}

\begin{document}

\newcommand{\dx}{\,\mathrm{d}x}
\newcommand{\dxi}{\,\mathrm{d}\xi}
\newcommand{\dy}{\,\mathrm{d}y}
\newcommand{\ds}{\,\mathrm{d}s}
\newcommand{\dt}{\,\mathrm{d}t}
\newcommand{\onehalf}{\nicefrac{1}{2}}
\newcommand{\norm}[1] {\| #1 \|}
\newcommand{\lrnorm}[1]{\left\| #1 \right\|}
\newcommand{\bignorm}[1]{\bigl\| #1 \bigr\|}
\newcommand{\Bignorm}[1]{\Bigl\| #1 \Bigr\|}
\newcommand{\Biggnorm}[1]{\Biggl\| #1 \Biggr\|}
\newcommand{\biggnorm}[1]{\biggl\| #1 \biggr\|}
\newcommand{\NN}{\mathbb{N}}
\newcommand{\ZZ}{\mathbb{Z}}
\newcommand{\RR}{\mathbb{R}}

\newcommand{\CHECKPROOF}[1]{ {\color{black} #1}  }


\title{Observability of a 1D Schrödinger equation with time-varying boundaries}
\author{Duc-Trung Hoang \thanks{\noindent Univ. Bordeaux, Institut de Math\'ematiques (IMB). CNRS UMR 5251. 351,  Cours de la Lib\'eration 33405 Talence, France. Email: dthoang@math.u-bordeaux.fr }}

\date{\today} 
\allowdisplaybreaks
\date{}

\maketitle

\abstract{ We discuss the observability of a one-dimensional
  Schrödinger equation on certain time dependent domain. In linear
  moving case, we give the exact boundary and pointwise internal
  observability for arbitrary time. For the general moving, we provide
  exact boundary observability when the curve satisfies some certain
  conditions . By duality theory, we establish the controllability of
  adjoint system. }

\vspace{.7cm}
\section{Introduction}
Let $\tau >0$, and $\ell(t) : [0,\tau] \rightarrow \mathbb{R}_{+}$ a
strictly positive $C^2$--function satisfying $\ell(0) = 1$ and
$\tfrac{\ell'}{\ell} \in L_\infty$. We consider the following system
as a initial boundary value problem in a time dependent domain.
\begin{equation} \label{initial-system-P}\tag{$S_{\text{moving}}$}
  \left\{
    \begin{array}{ll}
    i\tfrac{\partial u}{\partial t} + \tfrac{\partial^2 u}{\partial x^2} = 0  \qquad                        & x \in [0, \ell(t)]\\
    u(0, t) = u(\ell(t), t) = 0 \qquad        & t \ge 0\\
    u(x,0) = u_0              \qquad                       & x \in [0,1]\\
    \end{array}
  \right.
\end{equation}
For Neumann boundary observations we obtain estimates like
\[
            c(\tau) \; \norm{u_0}_{H_0^1(0,1)}^2
\; \le \;    \int_{0}^{\tau}|u_x(0,t)|^2 + |u_x(\ell(t),t)|^2 \,dt 
\; \le \;   C(\tau)\; \norm{u_0}_{H_0^1(0,1)}^2,
\]
see Theorems~\ref{thm:adms-ineq-general-case},
\ref{thm:obs-inequality} and \ref{thm:new-obs-boundary}.  We refer to
the first estimate as observability estimate and to the second as
admissibility estimate. The two first mentioned results rely on a
transformation of (\ref{initial-system-P}) to a non-autonomous
equation on the fixed domain $[0,1]$: the change of variables
$y =\frac{x}{\ell(t)}$ and new function $w(y,t) := u(x,t)$ gives an
equivalent differential equation for $w$, namely
\begin{equation}\label{transform-eq}\tag{$S_{\text{fixed}}$}
 \left\{
    \begin{array}{ll}
  i\frac{\partial w}{\partial t} & = \frac{-1}{\ell(t)^2}\frac{\partial^2 w}{\partial y^2} + i\frac{\ell'(t)}{\ell(t)}y\frac{\partial w}{\partial y},\\
w(0,t)  & = w(1, t)  = 0 \\
w_y(0,t) & = \ell(t)u_x(0,t) \quad\text{and}\\
 w_y(1,t) & = \ell(t)u_x(\ell(t),t)
    \end{array}\right.
  \end{equation}
 which can easily obtained by the chain rule.

 To obtain Theorems~\ref{thm:adms-ineq-general-case} and
 \ref{thm:obs-inequality} we apply the 'multiplier technique': This
 powerful method has been developped by Morawetz \cite{morawetz} and
 was later extended by Ho \cite{Ho} and Lions \cite{Lions}. We extend
 a version of Machtyngier \cite{elaine} to time-dependend
 multipliers. The observability estimate relies then on the
 ``uniqueness-compacity'' lemma~\ref{lem:unique-compactness}. The
 pitfall of this proof strategy is that it only proves existence of some positive 
 constant, without explicit estimates. 
 This is in contrast with Theorem~\ref{thm:new-obs-boundary} which is
 as specific result for the boundary curve
 $\ell(t) = 1{+}\varepsilon t$.  In this linear moving wall case, we
 mimic a successful approach for a one-dimensional wave-equation
 obtained by Haak and the author in \cite{haak} and develop the
 solution of (\ref{initial-system-P}) into a series of
 eigenfunctions. This allows to use results from Fourier analysis; the
 obtained admissibility estimates are sharper than those obtained in
 the previous results, and the observation estimate is provided with
 explicit constants.
 Moreover, we obtain in this case admissibility and
 exact observability of internal point observations:
\[
 k(\tau)\norm{ u_0 }_{L_2(0,1)}^2 
\le  \int_{0}^{\tau}|u(a,t)|^2 \dt 
\le  K(\tau)\norm{ u_0 }_{L_2(0,1)}^2,
\]
see Theorem~\ref{thm:observation-at-internal-points}. It is remakable
that the lower estimte cannot be true when $\varepsilon=0$ on any
rational point $a$ ; the fact that the considered domains extend
however, seem to 'middle out' this obstacle.  Closely related to this
observation are works of Castro and Khapalov
\cite{Castro:moving-interior,Khapalov:1995,Khapalov:2001} where on a
fixed domain $\Omega$ a moving point observer is considered, with
similar conclusions. We also mention results from Moyano
\cite{Moyano:PhD, Moyano:2D-circle} where in a two-dimensional circle
the radius $\ell(t)$ is used as a control parameter.

\smallskip

An additional result on $L_p$-admissibility and observability of point
observations are presented as well, see Theorem~\ref{thm:Lp-ell2-observation}.

\medskip

It is well-known that exact observability for an (autonomous) wave
equation implies observability for the associates Schrödinger
equation, see e.g.  \cite[Chapter 6.7 ff.]{TucsnakWeiss:book}. 
An inspection of the proof gives several obstacles when one passes to
non-autonomous problems, and we were not able to use this approach to
directly infer our results from those for the wave equation in
\cite{haak}. We mention that some results on the so-called Hautus-test
will be subject of an independent publication \cite{HHO}.

\section{Main Results}
Before giving precise formulations of the aforementioned results, let
us start by proving that the Schrödinger equation (\ref{transform-eq})
admits a solution: to this end, we reformulate it as an abstract
non-autonomous Cauchy problem in the following way: let $X = L_2(0,1)$
and the family of operators $\{A(t)\}$ be defined as
\begin{equation}  \label{eq:abstract-nonaut-pb}
   A(t)w = \frac{i}{\ell(t)^2} w_{yy} +\frac{\ell'(t)}{\ell(t)}yw_y
\end{equation}
wich natural domain 
$ D(A(t)) = H^2(0,1) \cap H_0^1(0,1) =: D $. Moreover, by assumption, 
the map $t \mapsto A(t)u$ is continuously differentiable for all $u \in D$.
Let $\omega >0$. Then integration by parts gives
\begin{equation}\label{operator coercive}
\begin{split}
\bigl\langle (A(t)+\omega I)w,w\bigr\rangle 
=  & \;\int_{0}^{1}\Bigl(\tfrac{i}{\ell(t)^2}w_{yy}\overline{w}+\tfrac{\ell'(t)}{\ell(t)}yw_y\overline{w}+\omega|w|^2 \Bigr)\,dy\\
= & \; \tfrac{-i}{\ell(t)^2}\int_{0}^{1}|w_y|^2\,dy +\tfrac{\ell'(t)}{\ell(t)}\int_{0}^{1}yw_y\overline{w}\,dy + \omega\int_{0}^{1}|w|^2\,dy\\
=  & \; \tfrac{-i}{\ell(t)^2}\int_{0}^{1}|w_y|^2\,dy -\tfrac{\ell'(t)}{\ell(t)}\int_{0}^{1}\bigl(|w|^2+yw\overline{w}_y\bigr)\,dy + \omega\int_{0}^{1}|w|^2\,dy
\end{split}
\end{equation}
Taking real parts and observing that
\[
  \text{Re}\Bigl(\int_{0}^{1}y w \overline{w_y}\,dy\Bigr)
= \text{Re}\Bigl(\tfrac{\ell'(t)}{\ell(t)}\int_{0}^{1}yw_y\overline{w}\,dy\Bigr) 
= -\text{Re} \Bigl(\tfrac{\ell'(t)}{2\ell(t)}\int_{0}^{1}|w|^2\,dy\Bigr)
\]
we obtain
\begin{equation}\label{operator coercive 2}
\text{Re}\Bigl(\bigl\langle (A(t)+\omega I)w,w\bigr\rangle\Bigr) = \Bigl(\omega - \tfrac{\ell'(t)}{2\ell(t)}\Bigr)\int_{0}^{1}|w|^2\,dy
\end{equation}
For $\omega > \bignorm{\tfrac{\ell'}{2\ell}}_{L^{\infty}}$, the left
hand side of (\ref{operator coercive 2}) becomes positive, and the
Lumer-Philips theorem asserts that $\omega+A(t)$ generates a contraction semigroup, i.e.
\[
\forall t \ge 0 \qquad  \bignorm{e^{-s \, A(t)}} \leq e^{\omega s}
\]
This ensures in particular that the family $( A(t) )_{t\in[0, \tau]}$
satisfies the Kato stability condition. We apply \cite[Theorem V.4.8 pp.145]{Pazy} to 
conclude that $(A(t))$ generates a unique evolution family
$\{U(t,s)\}_{0 \leq s \leq t \leq \tau}$ on $X$ satisfying  $w(t) =
U(t,0)w_0$. From this we infer a solution to 
(\ref{initial-system-P}) as well, by transforming the fixed domain 
back to the time-dependent domain.

\bigskip

Suppose that
we are given observation operators $C(t) : D \to Y$ where $Y$ is
another Hilbert space. Define the output function $y(t)=C(t)w(t)$. The
operator $C(t)$ is called $(Y, Z)$-admissible if there exist $\gamma
>0$ such that:
\[
\int_0^\tau \bignorm{ C(t) w(t) }_Y^2  dt \; \leq \; \gamma\, \norm{ w_0 }_Z^2.
\]
We say that the system (\ref{transform-eq}) is exactly $(Y, Z)$-observable in
time $\tau>0$ if there exist $\delta >0$ such that:
\[
  \int_0^\tau \bignorm{ C(t) w(t) }_Y^2  dt \; \ge \; \delta\, \norm{ w_0 }_Z^2.
\]
If the spaces $Y, Z$ are fixed, we simply speak of admissibility and
exact observability.  Exact observation in time $\tau>0$ means that
the knowledge of $y_{[0, \tau]}$ allows to recover the initial value
$w_0$. It is well known that exact observability is equivalent to
exact controllability of the retrograde adjoint system:
\[
  z'(t) = -A(t)^* z(t) - C(t)^* w(t) \qquad \text{with}\qquad z(\tau)=0
\] 
  Moreover, it is easy to see that admissibility or observability of
  (\ref{transform-eq}) is equivalent to those of
  (\ref{initial-system-P}).

\subsubsection*{Results on Neumann observations}

\begin{theorem}\label{thm:adms-ineq-general-case}
  Let $\tau > 0$ and $\ell : [0, \tau] \to \RR_+^*$ be a strictly
  positive, twice continuously differentiable function satisfying
  $\tfrac{\ell'}{\ell} \in L_\infty$ and $\ell(0)=1$.  Then there
  exists a constants $C(\tau)$ such that the following
  admissibility inequalities hold:
\[
\int_{0}^{\tau}|u_x(0,t)|^2 + |u_x(\ell(t),t)|^2 \,dt \quad \le \quad
\; C(\tau)\; \norm{u_0}_{H_0^1(0,1)}^2
\]
An explicit estimate of constant $C(\tau)$ is given in the proof, see
(\ref{eq:C1-computation}).
\end{theorem}

Concerning observability, we will have the following result. 
Let $\tau > 0$ and $\ell : [0, \tau] \to \RR_+^*$ be a strictly
positive, twice continuously differentiable function satisfying:
\begin{equation}\label{condition on the curve}
\ell'(t) > 0, \qquad \ell(0) = 1 \qquad \text{and} \qquad \ell'(t)\ell(t) < \frac{1}{\pi} \qquad \forall t \in (0,\tau)
\end{equation}
Integrating for $0$ to $\tau$ of the second condition, we have
$2\tau+\pi(1-\ell(\tau)^2) > 0$.
From the condition $(\ref{condition on the curve})$, $\ell(t)$ is
an increasing function, and then $\ell'(t) < \tfrac{1}{\pi}$. It 
follows that $\frac{\ell'(t)}{\ell(t)} < \frac{1}{\pi}$, and so 
the condition $\tfrac{\ell'}{\ell} \in L_\infty$ guaranteeing 
admissibility is satisfied.

\begin{theorem}\label{thm:obs-inequality}
  For all $\tau$ satisfying (\ref{condition on the curve}), the
  following observability inequality holds:
\[
c(\tau)\, \norm{u_0}_{H_0^1(0,1)}^2  \quad\le\quad   \int_{0}^{\tau}\bigl(|u_x(0,t)|^2+|u_x(\ell(t),t)|^2\bigr) dt.   
\]
Here $c(\tau)$ is some positive constant depending on $\tau$. 
\end{theorem}

A direct application of theorem~\ref{thm:obs-inequality} can be used
for periodic moving boundary $\ell(t) = 1+\varepsilon \sin(\omega t)$ where $\varepsilon \in (0,1)$ and $\omega \in (0,\frac{1}{\pi\varepsilon(1+\varepsilon)})$. For all $\tau \in \Bigl(0,\frac{\pi}{2\omega}\Bigr)$, we have

\CHECKPROOF
{
\[
\ell'(t) = \varepsilon\omega \cos(\omega t) > 0 \quad \text{since} \quad \omega t \in \Bigl(0,\frac{\pi}{2}\Bigr) \quad \forall 0 \leq t \leq \tau
\]
\[
\ell(0) = 1 \quad \text{and} \quad \ell'(t)\ell(t) = \varepsilon\omega\cos(\omega t)(1+\varepsilon\sin(\omega t)) < \varepsilon\omega(1+\varepsilon) < \frac{1}{\pi}
\]
}
Hence, $\ell(t)$ satisfies the condition (\ref{condition on the curve}), so the curve is admissible. The problem of particles moving inside one dimensional square-well of
oscillating width was proposed by Fermi and Ulam \cite{Fermi} in order
to explain the mechanism of particles containing high energies. This
model that plays an important role on theory of quantum chaos and it
seems difficult to give an exact solution formula. Glasser
\cite{glasser} investigated the behavior of wave functions and energy
in a given instantaneous eigenstate by assumptions on the smoothness
of boundary. As far as we know, there are no results in the
literature concerning observability and controllability with periodic
boundary functions.

In the case that $\ell(t) = 1{+}\varepsilon t$, the condition
$(\ref{condition on the curve})$ is ensured when
$\varepsilon \in (0,\tfrac{2}{\pi})$ and
$0 < t < \tfrac{1}{\varepsilon}\bigl(\tfrac{2}{\varepsilon\pi}-1\bigr)$.
We have the following exact analytic solution for
\ref{initial-system-P}, due to Doescher and Rice \cite{doescher}
 \begin{equation}\label{represetiation-formula}
 u(x,t) = \sum_{n=1}^{+\infty}  a_n \sqrt{\tfrac{2}{\ell(t)}}  \sin\bigl(\tfrac{n\pi x}{\ell(t)}\bigr)  e^{i (\frac{\varepsilon x^2}{4\ell(t)}-n^2\pi^2\tfrac{t}{\ell(t)}) }
 \end{equation}
 where the coefficients $(a_n)$ are defined by the sine-series
 development of the initial value $u_0$.  A similar exact solution in
 the case of two-variable moving wall can be found in \cite{yilmaz}
 where the author uses the fundamental transformation to change the
 moving boundary problem into a solvable one side fixed boundary
 problem.

\medskip

Based on formula (\ref{represetiation-formula}) we obtain a first result on Neumann
observability at the boundary $\{ (x, t): x \in \{ 0, \ell(t) \} \}$.
Compared to Theorem~\ref{thm:obs-inequality} the admissibility
constant is sharper. In contrast with 
Theorem~\ref{thm:obs-inequality}, where we can only prove existence of some positive
constant $c(\tau)$, we obtain now an explicit estimate for the
observability constant. The proof is presented in
section~\ref{sec:proof-linear-moving}.

\begin{theorem}\label{thm:new-obs-boundary}
  For every $\tau > 0$ there exist explicit  constants
  $c(\tau, \varepsilon), C(\tau, \varepsilon)$ such that:
  \begin{equation}\label{eq::Neumann-obs}
c(\tau,\varepsilon)\norm{ u_0 }_{H_0^1(0,1)}^2 \; \le \;
 \int_{0}^{\tau}\bigl|u_x(0,t)\bigr|^2 + \bigl|u_x(\ell(t),t)\bigr|^2\dt   \; \le \; 
C(\tau,\varepsilon)\norm{ u_0 }_{H_0^1(0,1)}^2
  \end{equation}
  In particular, the Neumann observation at the boundary of the system
  \eqref{initial-system-P} is exact observable in any time $\tau
  >0$.
  Moreover, the observability coefficient $c(\tau,\varepsilon)$ decays
  $\thicksim \exp\Bigr(\frac{-2k\pi^2}{\varepsilon\tau}\Bigr)$ where
  $k > \frac{3}{2}$.
\end{theorem}

\begin{remark} 
  By Dirichlet condition $u(\ell(t),t) = 0$ for all
  $t$. Differentiating yields
  $\ell'(t) u_x(\ell(t),t) + u_t(\ell(t),t) = 0$, and so
  $u_x(\ell(t),t) = \tfrac{-1}{\varepsilon}u_t(\ell(t),t)$. As a
  result, observing $u_t(\ell(t),t)$ or $u_x(\ell(t), t)$ is, up to a
  constant, the same.
\end{remark}

\subsubsection*{Point observations} 
We now focus on point observations $u \mapsto u(a, t)$ in the case of a linearly moving wall 
$\ell(t)=1{+}\varepsilon t$.
Observe that in the ``degenerate'' case that is, $\varepsilon = 0$, the (then)
autonomous Schrödinger equation has the well-known solution
\[
   u(x,t) = \sum_{n=1}^{+\infty}a_n e^{-i\pi^2n^2t}\sin(n\pi x).
\]
Clearly, there is no reasonable observability possible at rationals
points $x$ since infinitely many terms in the sum vanish,
independently of the leading coefficient $a_n$.  This changes when
$\varepsilon > 0$ : from (\ref{represetiation-formula}) we obtain
\[
  u(a,t) = \sum_{n=1}^{+\infty}a_n \bigl(\tfrac{2}{\ell(t)}\bigr)^{\tfrac{1}{2}}\exp\bigl(\tfrac{i\varepsilon a^2}{4\ell(t)}-in^2\pi^2\tfrac{t}{\ell(t)}\bigr)\sin\bigl(\tfrac{n\pi a}{\ell(t)}\bigr)
\]
and so
\begin{equation}  \label{eq:interal-point}
   \int_{0}^{\tau}\bigl|u(a,t)\bigr|^2 \dt = \int_{0}^{\tau} \tfrac{2}{\ell(t)}\Bigl| \sum_{n=1}^{+\infty}a_n e^{-i\pi^2n^2\tfrac{t}{\ell(t)}}\sin\bigl(\tfrac{n\pi a}{\ell(t)}\bigr)\Bigr|^2 \dt.
\end{equation}
Based on a remarkable result of Tenenbaum and Tucsnak we obtain the
following result in section~\ref{sec:proof-linear-moving}.
\begin{theorem}\label{thm:observation-at-internal-points}
Assume $\ell(t)=1{+}\varepsilon t$. Then, for every $\tau > 0$, we have:
\begin{equation}\label{hypothese}
\quad K(\tau)\norm{ u_0 }_{L_2(0,1)}^2 \quad \gtrsim \quad \int_{0}^{\tau}|u(a,t)|^2 \dt \quad \gtrsim \quad k(\tau)\norm{ u_0 }_{L_2(0,1)}^2
\end{equation}
More precisely, $k(\tau) \approx Me^{-\frac{c}{T}}$ where
$T = \frac{1}{\ell(0)} - \frac{1}{\ell(\tau)}$ and $M, c$ are some
positive constants that appear in to proof.
\end{theorem}


\begin{corollary}
  For all $a \in (0,1)$ the point observation $C=\delta_a$ for the
  system \eqref{initial-system-P} is exactly observable in arbitrary
  short time.
\end{corollary}

\subsubsection*{$L_p$-estimates  of point observations}
Finally we have to following $L_p$ admissibility and observability estimates.

\begin{theorem}\label{thm:Lp-ell2-observation}
  Let $\ell(t)=1{+}\varepsilon t$. We assume that
  $u_0 \in H_0^1(0,1)$. For $0 < p < 2$ and $a \in (0,1)$, we have
\[
k_p(\tau) \norm{u_0}_{L_2(0,1)}^{\nicefrac{2}{p}}  \norm{u_0}_{H_0^1(0,1)}^{1-\nicefrac{2}{p}} 
\le
\Bigl( \int_{0}^{\tau}\bigl|u(a,t)\bigr|^p dt\Bigr)^{\nicefrac{1}p}  
\le
K_p(\tau) \norm{u_0}_{L_2(0,1)}^{\nicefrac{2}{p}}  \norm{u_0}_{H_0^1(0,1)}^{1-\nicefrac{2}{p}} 
\]
where $k_p(\tau), $ are constants depending on $\tau$ and $p$.
\end{theorem}

The upper estimate is a direct consequence of (\ref{hypothese}).
Indeed, by the continuity of the embeddings $H_0^1 \hookrightarrow L_2 \hookrightarrow L_p$ and the boundedness of $\ell(t)$ to obtain:
\[
\norm{u(a,t)}_{L_p} \lesssim \norm{u(a,t)}_{L_2}  \overset{ \text{from} (\ref{hypothese})}{\lesssim} \norm{u_0}_{L_2} \lesssim \norm{u_0}_{L_2(0,1)}^{\nicefrac{2}{p}}  \norm{u_0}_{H_0^1(0,1)}^{1-\nicefrac{2}{p}} 
\]
Hence, it serves only to show that the lower estimate is of the right order. 


\section{Proof of the main results}\label{sec:proof-linear-moving}

\subsection{The multiplier Lemma}
 We follow E. Machtyngier \cite[Lemma 2.2]{elaine} by using multiplier method for
 (\ref{transform-eq}): Let $w$ be a solution to (\ref{transform-eq}) and
 $q \in C^2([0,1]\times [0, \tau])$ be a real valued
 function. Then, due to the differential equation (\ref{transform-eq}),
 \begin{equation}\label{multiplying-q}
   \text{Re}\Bigl(\int_{0}^{\tau}\int_{0}^{1}(q\overline{w_y}+\tfrac12\overline{w}q_y)\bigl(iw_t + \frac{1}{\ell(t)^2}w_{yy}-i\frac{\ell'(t)}{\ell(t)}yw_y \bigr) \,dy \,dt \Bigr) = 0 
 \end{equation}
  We separate the left hand side of (\ref{multiplying-q}) into three parts and simplify each of them.
 \begin{lemma}\label{lemma-multiply} The following identities hold.
   \begin{align}
      \label{item:one}
       &  \left\{\begin{array}{rl} 
    \; \text{Re}\Bigl(\displaystyle \int\limits_{0}^{\tau} & \hspace*{-2ex}\displaystyle\int\limits_{0}^{1}(q\overline{w_y}+\tfrac12\overline{w}q_y) iw_t\,dy\,dt \Bigr) \\
 =  & \;  \text{Re}\Bigl(\displaystyle \int\limits_{0}^{1} \Bigl[\tfrac12iq\overline{w_y}w\Bigl]_{t=0}^{t=T} dy\Bigr) -\tfrac12 \text{Re}\Bigl(\displaystyle \int\limits_{0}^{\tau}\displaystyle \int\limits_{0}^{1}iwq_t\overline{w_y} \,dy \,dt\Bigr)
          \end{array}\right.\\
      \label{item:two} 
     &   \left\{\begin{array}{rl} \displaystyle
     \; \text{Re}\Bigl(\displaystyle \int\limits_{0}^{\tau} &  \hspace*{-2ex}\displaystyle\int\limits_{0}^{1}\frac{w_{yy}}{l(t)^2}(q\overline{w_y}+\tfrac12\overline{w}q_y) \,dy \,dt\Bigr) \\
   = & \;  \text{Re}\Bigl(\displaystyle \int\limits_{0}^{\tau}\frac{1}{2\ell(t)^2}(q(1,t)|w_y(1,t)|^2 - q(0,t)|w_y(0,t)|^2) \,dt \Bigr) \\
    & \; -  \text{Re}\Bigl(\displaystyle \int\limits_{0}^{\tau}\displaystyle \int\limits_{0}^{1}\frac{1}{\ell(t)^2}|w_y|^2q_y \,dy \,dt \Bigr) - \text{Re}\Bigl(\displaystyle \int\limits_{0}^{\tau}\displaystyle \int\limits_{0}^{1}\frac{w_y\overline{w}}{2\ell(t)^2}q_{yy} \,dy \,dt \Bigr)
         \end{array}\right.\\
    \label{item:three}
     &
       \left\{\begin{array}{rl}
          \; -\text{Re}\Bigl(\displaystyle \int\limits_{0}^{\tau} &  \hspace*{-2ex}\displaystyle \int\limits_{0}^{1}\frac{iy\ell'(t)}{\ell(t)}w_y(q\overline{w_y}+\tfrac12\overline{w}q_y) \,dy \,dt\Bigr) \\
         = & \; -\text{Re}\Bigl(\displaystyle \int\limits_{0}^{\tau}\displaystyle \int\limits_{0}^{1}\frac{iy\ell'(t)}{\ell(t)}q|w_y|^2 \,dy \,dt \Bigr) - \text{Re}\Bigl(\displaystyle \int\limits_{0}^{\tau}\displaystyle \int\limits_{0}^{1}\tfrac12\frac{iy\ell'(t)}{\ell(t)}w_y\overline{w}q_y\,dy\,dt\Bigr)
       \end{array}\right.
   \end{align}
 \end{lemma}
 \begin{proof}
   To prove (\ref{item:one}), we use integration by parts. Using
   $\overline{w}(0,t) = \overline{w}(1,t) = 0$, we have:
 \begin{align*}
 \tfrac12\text{Re}\Bigl(i \int_{0}^{\tau}\int_{0}^{1} q_y \cdot \overline{w}  w_t \,dy\,dt \Bigr)
 = & \; \tfrac12\text{Re}\Bigl(i \int_{0}^{\tau}\Bigl( \Bigl[\overline{w} w_t q \Bigl]_{y=0}^{y=1} - \int_{0}^{1} q \cdot (\overline{w_y}w_t+\overline{w}w_{ty})\,dy \bigr)\,dt \Bigr) \\
 = & \; -\tfrac12\text{Re}\Bigl(i \int_{0}^{\tau}\int_{0}^{1}q(\overline{w_y}w_t+\overline{w}w_{ty})\,dy \,dt \Bigr) 
 \end{align*}
 Therefore, the left hand side of (\ref{item:one}) equals
 \begin{align*}
    & \;  \text{Re}\Bigl(\int_{0}^{\tau}\int_{0}^{1}(q\overline{w_y}+\tfrac12\overline{w}q_y) iw_t \,dy\, dt \Bigr) 
  = \tfrac12\text{Re}\Bigl(i \int_{0}^{1} \int_{0}^{\tau} (w_t \cdot q\overline{w_y} - q\overline{w}w_{ty}) \,dy\,dt \Bigr) \\
 = & \; \tfrac12\text{Re}\Bigl(i \int_{0}^{1}( \Bigl[q\overline{w_y}w\Bigl]_{t=0}^{t=\tau} - \int_{0}^{\tau} w(q_t\overline{w_y}+q\overline{w_{yt}}\,dt)\,dy \Bigr) - \tfrac12\text{Re}\Bigl(\int_{0}^{\tau}\int_{0}^{1}qi\overline{w}w_{ty}) \,dy\,dt \Bigr)\\
 = & \; \tfrac12\text{Re}\Bigl(i \int_{0}^{1}( \Bigl[q\overline{w_y}w\Bigl]_{t=0}^{t=\tau}\,dy \Bigr) - \tfrac12\text{Re}\Bigl(\int_{0}^{1}\int_{0}^{\tau}iwq_t\overline{w_y} \,dy\,dt \Bigr)
 \end{align*}
 Here, we already use the fact that 
 \[
 -\text{Re}\Bigl(\int_{0}^{1}\int_{0}^{\tau}iqw\overline{w_{ty}}\Bigr) = \text{Re}\Bigl(\int_{0}^{1}\int_{0}^{\tau}iq\overline{w}w_{ty} \,dt\,dy\Bigr).
 \]
To prove (\ref{item:two}) we have 
 \begin{align*}
 \text{Re}\Bigl(\int_{0}^{\tau}\int_{0}^{1}\frac{w_{yy}}{\ell(t)^2}q\overline{w_y}) \,dy\,dt \Bigr)
 = & \; \text{Re}\Bigl(\int_{0}^{\tau}\int_{0}^{1} \tfrac{d}{dy}(|w_y|^2) \cdot \frac{1}{2\ell(t)^2}q \,dy \,dt \Bigr) \\
 = & \;  \text{Re}\Bigl(\int_{0}^{\tau}\frac{1}{2\ell(t)^2}(q(1,t)|w_y(1,t)|^2 - q(0,t)|w_y(0,t)|^2) \,dt \Bigr)\\
 - & \;  \text{Re}\Bigl(\int_{0}^{\tau}\int_{0}^{1}\frac{1}{2\ell(t)^2}q_y |w_y|^2 \,dy\,dt \Bigr) 
 \end{align*}
 since we use $\text{Re}(w_{yy}\overline{w_y}) = \text{Re}(\overline{w_{yy}}w_y)$. Again, integration by parts shows
 \begin{align*}
 & \; \text{Re}\Bigl(\int_{0}^{\tau}\int_{0}^{1}\frac{w_{yy}}{2\ell(t)^2}\overline{w}q_y \,dy\,dt\Bigr) = \text{Re}\Bigl(\int_{0}^{\tau}\int_{0}^{1}\frac{1}{2\ell(t)^2}\overline{w}q_y d(w_y) \,dt\Bigr) \\
 = & \; \text{Re}\Bigl(\int_{0}^{\tau} \Bigl( \Bigl[\frac{1}{2\ell(t)^2}\overline{w}q_y w_y \Bigl]_{y=0}^{y=1} \bigr) \,dt\Bigr) - \text{Re}\Bigl(\int_{0}^{\tau}\int_{0}^{1}\frac{1}{2\ell(t)^2}(\overline{w_y}q_y+\overline{w}q_{yy}) w_y \,dt\Bigr) \\
 = & \; - \text{Re}\Bigl(\int_{0}^{\tau}\int_{0}^{1}\frac{1}{2\ell(t)^2}(\overline{w_y}q_y+\overline{w}q_{yy}) w_y dt\Bigr) 
 \end{align*}
 Therefore we have:
 \begin{align*}
 & \; \text{Re}\Bigl(\int_{0}^{\tau}\int_{0}^{1}\frac{w_{yy}}{\ell(t)^2}(q\overline{w_y}+\tfrac12\overline{w}q_y)\,dy\,dt\Bigr) \\
 = & \; \text{Re}\Bigl(\int_{0}^{\tau}\frac{1}{2\ell(t)^2}(q(1,t)w_y^2(1,t) - q(0,t)w_y^2(0,t))\,dt\Bigr) - \text{Re}\Bigl(\int_{0}^{\tau}\int_{0}^{1}\frac{1}{\ell(t)^2} |w_y|^2q_y\,dy\,dt \Bigr)\\
 - & \;  \text{Re}\Bigl(\int_{0}^{\tau}\int_{0}^{1}\frac{w_y\overline{w}}{2\ell(t)^2}q_{yy}\,dy\,dt \Bigr)
 \end{align*}
 Hence, part  (\ref{item:two}) is proved. The last part is obvious.
 \end{proof}
Now summing up the three parts and using (\ref{multiplying-q}) yields 

\begin{proposition}\label{important-corollary}
For any real valued function $q \in C^2([0,1]\times [0, \tau])$ and a solution $w$ to (\ref{transform-eq}) we have
\begin{align*}
0
= & \; \text{Re}\Bigl(\int_{0}^{1}\tfrac{i}2 \Bigl[q\overline{w_y}w\Bigl]_{t=0}^{t=\tau} dy\Bigr) -\tfrac12 \text{Re}\Bigl(\int_{0}^{\tau}\int_{0}^{1}iwq_t\overline{w_y} \,dy \,dt\Bigr) \\
+ & \;  \text{Re}\Bigl(\int_{0}^{\tau}\frac{1}{2\ell(t)^2}(q(1,t)|w_y(1,t)|^2 - q(0,t)|w_y(0,t)|^2) \,dt \Bigr) \\
- & \;  \text{Re}\Bigl(\int_{0}^{\tau}\int_{0}^{1}\frac{1}{\ell(t)^2}|w_y|^2q_y \,dy\,dt \Bigr) - \text{Re}\Bigl(\int_{0}^{\tau}\int_{0}^{1}\frac{w_y\overline{w}}{2\ell(t)^2}q_{yy} \,dy \,dt \Bigr) \\
- & \; \text{Re}\Bigl(\int_{0}^{\tau}\int_{0}^{1}\frac{iy\ell'(t)}{\ell(t)}q|w_y|^2 \,dy\,dt \Bigr) - \text{Re}\Bigl(\int_{0}^{\tau}\int_{0}^{1}\tfrac12\frac{iy\ell'(t)}{\ell(t)}w_y\overline{w}q_y \,dy\,dt\Bigr) 
\end{align*}
\end{proposition}

\subsection{Energy estimates}
For a solution $w$ to (\ref{transform-eq}) we define the first and second energy as
\[
E(t) = \tfrac12\int_{0}^{1}|w(y,t)|^2 dy\qquad\text{and}\qquad F(t) = \tfrac12\int_{0}^{1}|w_y(y,t)|^2 dy
\]
respectively. 

\begin{lemma}\label{first-energy-decay}
  We have $\ell(\tau) E(\tau) = E(0)$.
\end{lemma}
\begin{proof}
Taking the derivative respected to $t$ and using \ref{transform-eq}, we have
\begin{align*}
\frac{d E(t)}{dt}
= & \; \frac{d}{dt} \tfrac{1}{2}\int_{0}^{1}|w(y,t)|^2  dy
=  \;  \tfrac{1}{2}\int_{0}^{1}(w_t\overline{w} + w\overline{w_t}) dy\\
= & \; \tfrac{1}{2}\int_{0}^{1}\bigl(\tfrac{i}{\ell(t)^2}w_{yy}+\tfrac{\ell'(t)}{\ell(t)}yw_y\bigr)\overline{w} + w\overline{\bigl(\tfrac{i}{\ell(t)^2}w_{yy}+\tfrac{\ell'(t)}{\ell(t)}yw_y\bigr)} \\
= & \; \tfrac{1}{2}\int_{0}^{1}\bigl(\tfrac{i}{\ell(t)^2}(w_{yy}\overline{w}-\overline{w_{yy}}w) + \tfrac{\ell'(t)}{\ell(t)}y(w_y\overline{w} + \overline{w_y}w) \bigr) dy 
\end{align*}
Now integration by parts gives
\begin{align*}
\int_{0}^{1} \tfrac{i}{\ell(t)^2}(w_{yy}\overline{w}-\overline{w_{yy}}w) \,dy 
= & \; \int_{0}^{1} \tfrac{i}{\ell(t)^2} \overline{w}d(w_y) - \int_{0}^{1}\tfrac{i}{\ell(t)^2}w d(\overline{w_y}) \\
= & \; \tfrac{i}{\ell(t)^2}\bigl( \Bigl[\overline{w}w_y\Bigl]_{y=0}^{y=1} - \int_{0}^{1}|w_y|^2 \bigr) - \tfrac{i}{\ell(t)^2}\bigl(\Bigl[w\overline{w_y}\bigl]_{y=0}^{y=1} - \int_{0}^{1}|w_y|^2 \bigr) \\
= & \; 0
\end{align*}
whereas
\begin{align*}
& \; \int_{0}^{1} \tfrac{\ell'(t)}{\ell(t)}y(w_y\overline{w} + \overline{w_y}w) \,dy 
= \int_{0}^{1}\tfrac{\ell'(t)}{\ell(t)}y\overline{w}d(w) - \int_{0}^{1}\tfrac{\ell'(t)}{\ell(t)}ywd(\overline{w}) \\
= & \; \Bigl[\tfrac{\ell'(t)}{\ell(t)}y\overline{w}w\Bigl]_{y=0}^{y=1} - \tfrac{\ell'(t)}{\ell(t)}\int_{0}^{1}(\overline{w}+y\overline{w_y})w \,dy +
       \Bigl[\tfrac{\ell'(t)}{\ell(t)}y\overline{w}w\Bigl]_{y=0}^{y=1} - \tfrac{\ell'(t)}{\ell(t)}\int_{0}^{1}(w+yw_y)\overline{w} \,dy \\
= & \; -\tfrac{2\ell'(t)}{\ell(t)}\int_{0}^{1}|w(y,t)|^2 dy - \int_{0}^{1} \tfrac{\ell'(t)}{\ell(t)}y(w_y\overline{w} + \overline{w_y}w) \,dy .
\end{align*}
Therefore, 
\[
\int_{0}^{1} \tfrac{\ell'(t)}{\ell(t)}y(w_y\overline{w} + \overline{w_y}w) dy = -\tfrac{\ell'(t)}{\ell(t)}\int_{0}^{1}|w(y,t)|^2 dy,
\]
so that
\[
\tfrac{d E(t)}{dt} = - \tfrac{1}{2} \int_{0}^{1}\tfrac{\ell'(t)}{\ell(t)}|w(y,t)|^2 dy = -\tfrac{\ell'(t)}{\ell(t)}E(t).
\]
Using $\ell(0)=1$, this implies easily $E(\tau) = \tfrac{E(0)}{\ell(\tau)}$.
\end{proof}

\begin{lemma}\label{second-energy-decay}
For all $\tau >0$ and $\tau \in \Bigl(0,\tfrac{\pi}{2\omega}\Bigr)$, we have:
\[
\tfrac{\pi^2}{\ell(\tau)}E(0)  \leq F(\tau) \leq \ell(\tau) F(0)
\]
\end{lemma}
\begin{proof}
Concerning $F$ we have 
\begin{align*}
\frac{d F(t)}{dt}
= & \; \frac{d}{dt} \tfrac12\int_{0}^{1}|w_y(y,t)|^2  
=  \;  \tfrac12\int_{0}^{1}(w_{yt}\overline{w_y} + w_y\overline{w_{yt}})\,dt\\
= & \; \tfrac12\int_{0}^{1}\bigl(\tfrac{i}{\ell(t)^2}w_{yy}+\tfrac{\ell'(t)}{\ell(t)}yw_y\bigr)_y\overline{w_y} + w_y\overline{\bigl(\tfrac{i}{\ell(t)^2}w_{yy}+\tfrac{\ell'(t)}{\ell(t)}yw_y\bigr)_y} \\
= & \; \tfrac{i}{2\ell(t)^2}\int_{0}^{1}(w_{yyy}\overline{w_y}-\overline{w_{yyy}}w_y)\,dy + \tfrac{\ell'(t)}{2\ell(t)}\int_{0}^{1}((yw_y)_y\overline{w_y} + w_y\overline{(yw_y)_y})\,dy .
\end{align*}
The first term on the right hand side simplifies as
\begin{align*}
  & \;  \tfrac{i}{2\ell(t)^2}\int_{0}^{1}(w_{yyy}\overline{w_y}-\overline{w_{yyy}}w_y) \,dy \\ 
= & \; \tfrac{i}{2\ell(t)^2}\int_{0}^{1}\overline{w_y}\,d(w_{yy}) - \tfrac{i}{2\ell(t)^2}\int_{0}^{1}w_y\,d(\overline{w_{yy}}) \\
= & \; \tfrac{i}{2\ell(t)^2} \Bigl[\overline{w_y}w_{yy}\Bigl]_{y=0}^{y=1}- \tfrac{i}{2\ell(t)^2}\int_{0}^{1}|w_{yy}|^2 \,dy - \tfrac{i}{2\ell(t)^2}\Bigl[\overline{w_{yy}}w_{y}\Bigl]_{y=0}^{y=1} + \tfrac{i}{2\ell(t)^2}\int_{0}^{1}|w_{yy}|^2 \,dy \\
= & \; \Bigl[\tfrac12\overline{w_y}(w_t-\tfrac{\ell'(t)}{\ell(t)}yw_y)\Bigl]_{y=0}^{y=1} 
     + \Bigl[\tfrac12w_y(\overline{w_t}-\tfrac{l'(t)}{l(t)}y\overline{w_y})\Bigl]_{y=0}^{y=1}  =  - \tfrac{\ell'(t)}{\ell(t)}|w_y(1,t)|^2
\end{align*}
whereas the second term simplifies as follows.
\begin{align*}
\tfrac{\ell'(t)}{2\ell(t)}\int_{0}^{1}((yw_y)_y\overline{w_y} + w_y\overline{(yw_y)_y})\,dy
= & \; \tfrac{\ell'(t)}{2\ell(t)}\int_{0}^{1}(w_y+yw_{yy})\overline{w_y}+w_y(\overline{w_y+yw_{yy}})\,dy \\
= & \; \tfrac{\ell'(t)}{2\ell(t)}\int_{0}^{1} 2|w_y|^2 + y(w_{yy}\overline{w_y}+w_y\overline{w_{yy}})\,dy \\
= & \; \tfrac{\ell'(t)}{\ell(t)}\int_{0}^{1} |w_y|^2\,dy +  \tfrac{\ell'(t)}{2\ell(t)}\int_{0}^{1} y\,d(|w_y|^2) \\
= & \; \tfrac{\ell'(t)}{2\ell(t)}\int_{0}^{1} |w_y|^2\,dy + \tfrac{\ell'(t)}{2\ell(t)}|w_y(1,t)|^2.
\end{align*}
We add both parts to obtain
\begin{align*}
\frac{d F(t)}{dt}
= & \; \tfrac{\ell'(t)}{2\ell(t)}\int_{0}^{1}|w_y(y,t)|^2\,dt   -  \tfrac12|w_y(1,t)|^2\tfrac{\ell'(t)}{\ell(t)} \\
= & \; \tfrac{\ell'(t)}{\ell(t)} \Bigl(F(t) -\tfrac12|w_y(1,t)|^2 \Bigr),
\end{align*}
By Variation of constants, we get an explicit solution:
\begin{equation}\label{equation for F}
F(t) = \ell(t)F(0) - \ell(t)\int_{0}^{t}\tfrac{\ell'(s)}{2\ell(s)^2}|w_y(1,s)|^2\,ds
\end{equation}
%
One easily obtains an upper bound, namely $F(t) \le F(0) \ell(t)$. For the lower bound,
we use the Poincaré (or Wirtinger) inequality on $[0,1]$ to obtain,
\begin{equation}\label{bounded-for-F}
  F(t) = \tfrac12 \int_{0}^{1}|w_y(y,t)|^2 \, dy \geq  \tfrac{\pi^2}{2}\int_{0}^{1}|w(y,t)|^2 \,dy = \tfrac{\pi^2}{\ell(t)}E(0) 
\end{equation}
\end{proof}

\subsection{Admissibility of Neumann observations at the boundary}
\begin{proof}[Proof of Theorem~\ref{thm:obs-inequality}]
We take the function $q(y,t) = q(y)$ on $(0,1)$ satisfying $q(1) = 0$ and
$q(0) = 1$. By Proposition~\ref{important-corollary}, we have
\begin{align*}
& \; \text{Re}\Bigl(\int_{0}^{\tau}\tfrac{1}{2\ell(t)^2}q(0,t)|w_y(0,t)|^2 \,dt\Bigr) = \text{Re}\Bigl(\int_{0}^{1}  \Bigr[\tfrac12iq\overline{w_y}w\Bigr]_{t=0}^{t=\tau} \,dy\Bigr) \\
- & \;\text{Re}\Bigl(\int_{0}^{\tau}\int_{0}^{1}\tfrac{1}{\ell(t)^2}|w_y|^2q_y\,dy\,dt \Bigr) - \text{Re}\Bigl(\int_{0}^{\tau}\int_{0}^{1}\frac{w_yw}{2\ell(t)^2}q_{yy} \,dy \,dt \Bigr) \\
- & \; \text{Re}\Bigl(\int_{0}^{\tau}\int_{0}^{1}\tfrac{iy\ell'(t)}{\ell(t)}q|w_y|^2 \,dy\,dt \Bigr) - \text{Re}\Bigl(\int_{0}^{\tau}\int_{0}^{1}\tfrac12\tfrac{iy\ell'(t)}{\ell(t)}w_y\overline{w}q_y \,dy\,dt\Bigr)
\end{align*}
Therefore, we have 
\[
  \int_{0}^{\tau}\frac{1}{2\ell(t)^2}|w_y(0,t)|^2 \,dt \leq A+B+C+D+E,
\]
where we estimate all five terms separately. Concerning $A$, we
separate the products in the real part by $ab \le \tfrac12(a^2+ b^2)$,
then use Lemmata~\ref{first-energy-decay} and
\ref{second-energy-decay} to obtain
\begin{align*}
& \;  A = \Bigl|\text{Re}\Bigl(\int_{0}^{1} \Bigr[\tfrac12iq\overline{w_y}w\Bigr]_{t=0}^{t=\tau} \,dy\Bigr| \\
\leq & \; \tfrac{1}{4} \norm{ q }_{L_{\infty}(0,1)}\Bigl(\int_{0}^{1}|w(y,\tau)|^2+|w(y,0)|^2+|w_y(y,\tau)|^2+|w_y(y,0)|^2 \,dy\Bigr) \\
= & \; \tfrac{1}{4} \norm{ q }_{L_{\infty}(0,1)}\Bigl(\int_{0}^{1}\Bigl(\tfrac{1}{\ell(\tau)}+1\Bigr)|w(y,0)|^2+(1+\ell(\tau))|w_y(y,0)|^2 \,dy\Bigr) \\
\leq & \; \tfrac{1}{4} \norm{ q }_{L_{\infty}(0,1)}\Bigl(\int_{0}^{1}\Bigl(\tfrac{\pi^2}{\ell(\tau)}+\pi^2 +1+\ell(\tau)\Bigr)|w_y(y,0)|^2 \,dy\Bigr).
\end{align*}
The second term is easily estimated by Lemma~\ref{first-energy-decay}:
\begin{align*}
B = \Bigl|\text{Re}\Bigl(\int_{0}^{\tau}\int_{0}^{1}\tfrac{1}{\ell(t)^2}|w_y(y, t)|^2q_y\,dy\,dt \Bigr)\Bigr| 
\leq & \; \norm{ q_y }_{L_{\infty}(0,1)}\int_{0}^{\tau}\int_{0}^{1}\tfrac{1}{\ell(t)^2}|w_y(y,t)|^2 \,dy\,dt \\
\leq & \;\norm{ q_y }_{L_{\infty}(0,1)}  \int_{0}^{\tau} \int_{0}^{1}\tfrac{1}{\ell(t)}|w_y(y,0)|^2 \dy\dt.
\end{align*}
Part $C$ is decoupled by Cauchy-Schwarz and then estimated using Lemma~\ref{second-energy-decay} as follows:
\begin{align*}
& \; C = \Bigl|\text{Re}\Bigl(\int_{0}^{\tau}\int_{0}^{1}\frac{w_yw}{2\ell(t)^2}q_{yy} \,dy \,dt \Bigr) \Bigr| \leq  \norm{ q_{yy} }_{L_{\infty}(0,1)}\Bigl(\int_{0}^{\tau}\int_{0}^{1}\frac{|w_yw|}{2\ell(t)^2}\,dy \,dt \Bigr) \\
\leq & \; \norm{ q_{yy} }_{L_{\infty}(0,1)}\Bigl(\int_{0}^{\tau}\tfrac{1}{2\ell(t)^2}\Bigl(\int_{0}^{1}|w(y,t)|^2\,dy\Bigr)^{{\nicefrac{1}{2}}}\Bigl(\int_{0}^{1}|w_y(y,t)|^2\,dy\Bigr)^{{\nicefrac{1}{2}}}\,dt\Bigr)\\
\leq & \; \norm{ q_{yy} }_{L_{\infty}(0,1)}\Bigl(\int_{0}^{\tau}\tfrac{\pi}{2\ell(t)^2}\Bigl(\int_{0}^{1}|w_y(y,t)|^2\,dy\Bigr)\,dt\Bigr)\\
\leq & \; \norm{ q_{yy} }_{L_{\infty}(0,1)}\Bigl(\int_{0}^{\tau}\tfrac{\pi}{2\ell(t)}\,dt\Bigr)\Bigl(\int_{0}^{1}|w_y(y,0)|^2\,dy\Bigr).
\end{align*}
For the forth part, we use Lemma~\ref{second-energy-decay} to obtain
\begin{align*}
D = \Bigl| \text{Re}\Bigl(\int_{0}^{\tau}\int_{0}^{1}\tfrac{iy\ell'(t)}{\ell(t)}q|w_y(y,t)|^2 \,dy\,dt \Bigr)\Bigr| \leq 
& \; \norm{q}_{L_{\infty}(0,1)}\int_{0}^{\tau}\int_{0}^{1}\tfrac{\ell'(t)}{\ell(t)}|w_y(y,t)|^2 \,dy\,dt \\
\leq & \; \norm{q}_{L_{\infty}(0,1)}\Bigl(\int_{0}^{\tau}\ell'(t)\,dt\Bigr)\Bigl(\int_{0}^{1}|w_y(y,0)|^2\,dy\Bigr) \\
= & \; \norm{q}_{L_{\infty}(0,1)}(\ell(\tau)-1)\Bigl(\int_{0}^{1}|w_y(y,0)|^2\,dy\Bigr).
\end{align*}
Finally, part $E$ is treated like part $C$:
\begin{align*}
& \; E = \Bigl| \text{Re}\Bigl(\int_{0}^{\tau}\int_{0}^{1}\tfrac12\frac{iy\ell'(t)}{\ell(t)}w_y\overline{w}q_y \,dy\,dt\Bigr)\Bigr| \leq \norm{ q_{y} }_{L_{\infty}(0,1)}\Bigl(\int_{0}^{\tau}\int_{0}^{1}\tfrac12\frac{\ell'(t)}{\ell(t)}|w_y||\overline{w}|\,dy\,dt\Bigr) \\
\leq & \; \norm{ q_{y} }_{L_{\infty}(0,1)}\Bigl(\int_{0}^{\tau}\tfrac{\ell'(t)}{2\ell(t)}\Bigl(\int_{0}^{1}|w(y,t)|^2\,dy\Bigr)^{{\nicefrac{1}{2}}}\Bigl(\int_{0}^{1}|w_y(y,t)|^2\,dy\Bigr)^{{\nicefrac{1}{2}}}\,dt\Bigr)\\
\leq & \; \norm{ q_{y} }_{L_{\infty}(0,1)}\Bigl(\int_{0}^{\tau}\tfrac{\pi\ell'(t)}{2\ell(t)}\Bigl(\int_{0}^{1}|w_y(y,t)|^2\,dy\Bigr)\,dt\Bigr)\\
\leq & \; \norm{ q_{y} }_{L_{\infty}(0,1)}\Bigl(\int_{0}^{\tau}\tfrac{\pi\ell'(t)}{2}\,dt\Bigr)\Bigl(\int_{0}^{1}|w_y(y,0)|^2\,dy\Bigr) \\
= & \; \norm{ q_{y} }_{L_{\infty}(0,1)}\tfrac{\pi}{2}(\ell(\tau)-1)\Bigl(\int_{0}^{1}|w_y(y,0)|^2\,dy\Bigr).
\end{align*}
Summing up all three estimates, we obtain
\begin{equation}\label{upper-admiss}
\int_{0}^{\tau}\frac{1}{2\ell(t)^2}\bigl|w_y(0,t)\bigr|^2 \,dt \leq C_1(\tau)\norm{ w_0 }_{H_0^1(0,1)}^2
\end{equation}
where the constant $C_1(\tau)$ is given by
\begin{equation}\label{eq:C1-computation}
\begin{aligned}
C_1(\tau) 
= & \; \frac{5\ell(\tau)^2+(\pi^2-3)\ell(\tau)+\pi^2}{4\ell(\tau)}\norm{q}_{L_{\infty}(0,1)}+\Bigl(\frac{\pi}{2}(\ell(\tau)-1)+\int_{0}^{\tau}\frac{dt}{\ell(t)}\Bigr)\norm{q_y}_{L_{\infty}(0,1)} \\
+ & \; \Bigl(\int_{0}^{\tau}\frac{\pi}{2\ell(t)}\,dt\Bigr)\norm{q_{yy}}_{L_{\infty}(0,1)}
\end{aligned}
\end{equation}
Replacing $w_y(0,t) = \ell(t)u_x(0,t)$ in $(\ref{upper-admiss})$ yields the admissibility inequality:
\[
\int_{0}^{\tau}\bigl|u_x(0,t)\bigr|^2 \,dt \leq 2C_1(\tau)\norm{ u_0 }_{H_0^1(0,1)}^2
\]
The second admissibility estimate follows the same lines, using $q(y,t) = q(y)$ on $(0,1)$ with $q(0) = 0$ and $q(1) = 1$.
\end{proof}


\subsubsection{Neumann Observability at the Boundary}
Recall the following lemma

\begin{lemma}\label{lem:unique-compactness}
Let $E_1, E_2$ and $E_3$ be the Hilbert spaces. We consider the continuous linear operators $T: E_1 \rightarrow E_2$, $K: E_1 \rightarrow E_3$ and $L: E_1 \rightarrow E_1$ such that $K$ is compact, $L$ is bounded below and:
\begin{equation}
\norm{Lu}_{E_1} \approx \norm{Tu}_{E_2}+\norm{Ku}_{E_3}
\end{equation}
Then the kernel of $A$ has finite dimension and $\norm{Lu}_{E_1} \approx \norm{Tu}_{E_3}$
\end{lemma}
\begin{proof}
A similar proof can be found in \cite[Lemma 1 pp.1]{Tartar} where we just replace $u$ by $Lu$.
\end{proof}

\begin{proof}[Proof of Theorem~\ref{thm:obs-inequality}]
 For all $\tau$ satisfying  $2\tau+\pi(1-\ell(\tau)^2) > 0$, we choose two positive constants $\eta(\tau) $ and $\delta(\tau)$ such that:
\begin{equation}\label{assumption-on-tau}
\eta(\tau) + \delta(\tau) < \tfrac{4}{1+\ell(\tau)^3}\Bigl(\tau-\tfrac{\pi}{2}(\ell(\tau)^2-1)\Bigr)
\end{equation}
 We choose $q(y) = (1-y)\ell(t)$ where $y \in (0,1)$. Proposition~\ref{important-corollary} is then equivalent to:
 \begin{equation}\label{q(y)-=-1-y}
\begin{aligned}
\int_{0}^{\tau}\frac{1}{2\ell(t)}\bigl|w_y(0,t)\bigr|^2 \,dt 
= & \; \int_{0}^{\tau}\int_{0}^{1}\frac{1}{\ell(t)}|w_y|^2 \,dy\,dt - \text{Re}\Bigl(\int_{0}^{\tau}\int_{0}^{1}\tfrac12 i(1-y)\ell'(t)\overline{w}_yw \,dy\,dt\Bigr)  \\
+ & \; \text{Re}\Bigl(\int_{0}^{1} \Bigr[\tfrac12i(1-y)\ell(t)\overline{w_y}w\Bigr]_{t=0}^{t=\tau} dy\Bigr)+\text{Re}\Bigl(\int_{0}^{\tau}\int_{0}^{1}\tfrac12 iy\ell'(t)w_y\overline{w} \,dy\,dt\Bigr) 
\end{aligned}
\end{equation}
Taking the three last formula of the right hand side to the left, then taking the absolute to get:
\begin{align*}
\int_{0}^{\tau}\int_{0}^{1}\frac{1}{\ell(t)}|w_y|^2 \,dy\,dt
\leq  & \; \int_{0}^{\tau}\frac{1}{2\ell(t)}\bigl|w_y(0,t)\bigr|^2 \,dt + \Bigl|\text{Re}\Bigl(\int_{0}^{1} \Bigr[\tfrac12i(1-y)\ell(t)\overline{w_y}w\Bigr]_{t=0}^{t=\tau} dy\Bigr)\Bigr| \\
+ & \; \Bigl|\text{Re}\Bigl(\int_{0}^{\tau}\int_{0}^{1}\tfrac12 i(1-y)\ell'(t)\overline{w}_yw \,dy\,dt\Bigr)\Bigr| \\
+ & \; \Bigl|\text{Re}\Bigl(\int_{0}^{\tau}\int_{0}^{1}\tfrac12 iy\ell'(t)w_y\overline{w} \,dy\,dt\Bigr)\Bigr|
\end{align*}
The sum of third and fourth terms in the right hand side of above formula can be estimated as:
\begin{align*}
& \; \Bigl|\text{Re}\Bigl(\int_{0}^{\tau}\int_{0}^{1}\tfrac12 i(1-y)\ell'(t)\overline{w}_yw \,dy\,dt\Bigr)\Bigr| + \Bigl|\text{Re}\Bigl(\int_{0}^{\tau}\int_{0}^{1}\tfrac12 iy\ell'(t)w_y\overline{w} \,dy\,dt\Bigr)\Bigr| \\
\leq  & \; \tfrac{1}{2}\int_{0}^{\tau}\int_{0}^{1}\ell'(t)|w_y\overline{w} |\,dy\, dt + \tfrac{1}{2}\int_{0}^{\tau}\int_{0}^{1}\ell'(t)|\overline{w}_yw|\,dy\,dt \\
\leq & \; \int_{0}^{\tau}\ell'(t)\Bigl(\int_{0}^{1}|w|^2 \,dy\Bigr)^{1/2}\Bigl(\int_{0}^{1}|w_y|^2 \,dy\Bigr)^{1/2} \,dt  \\
\leq & \; \int_{0}^{\tau}\pi\ell'(t)\Bigl(\int_{0}^{1}|w_y|^2 \,dy\Bigr)\,dt
\end{align*}
Due to the energy estimate in lemma \ref{first-energy-decay} and \ref{second-energy-decay}, we have the upper bound for the second term:
\begin{align*}
& \; \Bigl|\text{Re}\Bigl(\int_{0}^{1} \Bigr[\tfrac12i(1-y)\ell(t)\overline{w_y}w\Bigr]_{t=0}^{t=\tau}\,dy\Bigr)\Bigr| \\
\leq & \; \tfrac{1}{4}\int_{0}^{1}\Bigl(\frac{|w(y,0)|^2}{\eta(\tau)}+\frac{|w(y,\tau)|^2}{\eta(\tau)}+\eta(\tau)|w_y(y,0)|^2+\eta(\tau)\ell(\tau)^2|w_y(y,\tau)|^2\Bigr)\,dy \\
\leq & \; \tfrac{1}{4\eta(\tau)}\Bigl(\tfrac{1}{\ell(\tau)}+1\Bigr)\int_{0}^{1}|w(y,0)|^2\,dy + \tfrac{(1+\ell(\tau)^3)\eta(\tau)}{4}\int_{0}^{1} |w_y(y,0)|^2\,dy \\
\end{align*}
As a result, we combine these estimation and use (\ref{equation for F}) to obtain:
\begin{align*}
& \; \int_{0}^{\tau}\tfrac{1}{2\ell(t)}|w_y(0,t)|^2 \,dt + \tfrac{1}{4\eta(\tau)}\Bigl(\tfrac{1}{\ell(\tau)}+1\Bigr)\int_{0}^{1}|w(y,0)|^2 \,dy \\
\geq  & \; \int_{0}^{\tau}\int_{0}^{1}\Bigl(\tfrac{1}{\ell(t)}-\pi\ell'(t)\Bigr)|w_y(y,t)|^2 \,dy\,dt - \tfrac{(1+\ell(\tau)^3)\eta(\tau)}{4}\int_{0}^{1} |w_y(y,0)|^2\,dy  \\
=  & \; \Bigl(\int_{0}^{\tau}(1-\pi\ell'(t)\ell(t))\,dt-\tfrac{(1+\ell(\tau)^3)\eta(\tau)}{4}\Bigr)\Bigl(\int_{0}^{1}|w_y(y,0)|^2\,dy\Bigr) \\
 - & \; \int_{0}^{\tau}\bigl(1-\pi\ell'(t)\ell(t)\bigr)\int_{0}^{t}\tfrac{\ell'(s)}{\ell(s)^2}|w_y(1,s)|^2\,ds\,dt\\
 = & \; \Bigl(\tau+\tfrac{\pi}{2}(1-\ell(\tau)^2)-\tfrac{(1+\ell(\tau)^3)\eta(\tau)}{4}\Bigr)\Bigl(\int_{0}^{1}|w_y(y,0)|^2\,dy\Bigr) \\
 - & \; \int_{0}^{\tau}\bigl(1-\pi\ell'(t)\ell(t)\bigr)\int_{0}^{t}\tfrac{\ell'(s)}{\ell(s)^2}|w_y(1,s)|^2\,ds\,dt\\
 \geq & \; \tfrac{(1+\ell(\tau)^3)\delta(\tau)}{4}\Bigl(\int_{0}^{1}|w_y(y,0)|^2\,dy\Bigr) - \int_{0}^{\tau}\bigl(1-\pi\ell'(t)\ell(t)\bigr)\int_{0}^{t}\tfrac{\ell'(s)}{\ell(s)^2}|w_y(1,s)|^2\,ds\,dt
\end{align*}
where the last inequality come from (\ref{assumption-on-tau}). Therefore, there exist the constants $A_{\tau}$ and $B_{\tau}$ such that:
\begin{equation}\label{estimate operators 1}
\int_{0}^{1} |w_y(y,0)|^2\,dy  \leq A_{\tau}\int_{0}^{\tau}\bigl(|w_y(0,t)|^2+|w_y(1,t)|^2\bigr)\,dt+B_{\tau}\int_{0}^{1} |w(y,0)|^2\,dy
\end{equation}
It is sufficient to prove that there exist a constant $K >0$ such that
\begin{equation}\label{machtyginer-argument}
\int_{0}^{1}|w(y,0)|^2\,dy  \leq K \Bigl(\int_{0}^{\tau} |w_y(0,t)|^2\,dt + \int_{0}^{\tau}|w_y(1,t)|^2\,dt\Bigr)
\end{equation}
Let us denote the operator $T$ from $H_0^1(0,\tau)$ to $L_2(0,\tau)\times L_2(0,\tau)$ and the operator $K$ from $H_0^1(0,1)$ to $L_2(0,1)$ that maps:
\begin{equation}
(Tw)(t) = \bigl(w_y(0,t),w_y(1,t)\bigr)
\end{equation}
\begin{equation}
(Kw)(y) = w(y,0)
\end{equation}
From admissibility and (\ref{estimate operators 1}), we have:
\begin{equation}
a_{\tau}\norm{Tw}_{L_2}^2 + b_{\tau}\norm{Kw}_{L_2}^2 \leq \norm {w_0}_{H_0^1}^2 \leq A_{\tau}\norm{Tw}_{L_2}^2 + B_{\tau}\norm{Kw}_{L_2}^2
\end{equation}
It is easy to see that $K$ is compact operator due to Rellich’s embedding lemma. 
In order to use the unique-compactness lemma~\ref{lem:unique-compactness} for $L=K$, we need to check 
that $T$ is injective.  Observe that $T w=0$ means that $w$ satisfies (\ref{transform-eq})
with Dirichlet conditions and zero Neumann derivative.
It is well known that $w$ vanishes in this case,  see for example
\cite[Theorem~3]{Tataru:schroedinger} or \cite[Corollary~6.1]{Isakov}.
As a consequence, 
\[
c_{\tau}\norm{Tw}_{L_2}^2 \le  \norm {w_0}_{H_0^1}^2 \le C_{\tau}\norm{Tw}_{L_2}^2
\]
for some constants $c(\tau), C(\tau)>0$.
\end{proof}


\subsection{Results for linear moving walls}
Recall the Doescher-Rice representation formula~(\ref{represetiation-formula}) that yields for $t=0$ \begin{equation}\label{representation-u(x,0)}
   u(x, 0) = \sqrt{2} \sum_{n=1}^{N}a_n e^{\tfrac{i\varepsilon x^2}{4}}\sin(n\pi x),
\end{equation}
and denote by
\[
    u_n(x, t) :=  \sqrt{\tfrac{2}{\ell(t)}}  \sin\bigl(\tfrac{n\pi x}{\ell(t)}\bigr).
\]
For all fixed $t >0$, the functions $(u_n(\cdot,t))_{n\ge 1}$ form an
orthonormal basis in $L_2(0,\ell(t))$, since the change of variable
$y = \tfrac{x}{\ell(t)}$ reduces $u_n(\cdot, t)$ to the standard
trigonometric system on $L_2([0,1])$.

\begin{lemma}\label{initial-norm}
  For all finitely supported sequences $(a_n)$ we have the following
  relation between $(a_n)$ and the norms of the initial data $u_0$.
\[ 
\norm{ u(x,0) }_{L_2(0,1)}^2 = \sum_{n=1}^{+\infty}|a_n|^2, \quad \quad \norm{ u(x,0) }_{H_0^1(0,1)}^2 \thicksim \sum_{n=1}^{+\infty}|a_n|^2 n^2
\]
\end{lemma}
\begin{proof}
Observe that 
\begin{align*}
\norm{e^{-\tfrac{i\varepsilon x^2}{4}}u_N(x) }_{L_2(0,1)}^2
= & \; \norm{ u_N(x) }_{L_2(0,1)}^2
=  2\int_{0}^{1}\Bigl|\sum_{n=1}^{N}a_n \sin(n\pi x)\Bigr|^2 \dx \\
= & \; \; 2\int_{0}^{1}\Bigl|\sum_{n=1}^{N}a_n \sin(n\pi x)\Bigr|^2 \dx 
=  \sum_{n=1}^{\infty}|a_n|^2.
\end{align*}
Since $(a_n)$ is a finite sequence we may interchange differentiation and summation and obtain
\[
\tfrac{\mathrm{d}}{\dx}
u(x) = \sqrt{2} \sum_{n=1}^{N} a_n e^{\tfrac{i\varepsilon x^2}{4}} 
(ix \tfrac{\varepsilon}2  \sin(n\pi x) + n\pi\cos(n\pi x) )
\]
so that, squaring real and imaginary parts, we find
\begin{align*}
\norm{u(x) }_{H_0^1(0,1)}^2
= & \; 2\int_{0}^{1}\Bigl|\sum_{n=1}^{N}a_n n\pi \cos(n\pi x)\Bigr|^2 \dx 
      + 2 \int_{0}^{1}\Bigl|\sum_{n=1}^{N}a_n x \tfrac{\varepsilon}2  \sin(n\pi x)\Bigr|^2\\
= & \; \pi^2 \sum_{n=1}^{N}|a_n|^2n^2 + 2 \int_{0}^{1}\Bigl|\sum_{n=1}^{N}a_n x \tfrac{\varepsilon}2  \sin(n\pi x)\Bigr|^2 \\
\leq & \; \pi^2 \sum_{n=1}^{N}|a_n|^2n^2 + \frac{\varepsilon^2}{2} \int_{0}^{1}\Bigl|\sum_{n=1}^{N}a_n  \sin(n\pi x)\Bigr|^2 \\
= & \; \pi^2 \sum_{n=1}^{N}|a_n|^2n^2 + \frac{\varepsilon^2}{2}\sum_{n=1}^{N}|a_n|^2 \leq  C(\varepsilon) \sum_{n=1}^{N}|a_n|^2n^2  \qedhere
\end{align*}
\end{proof}

\begin{lemma}\label{ONS-lemma}
  Let $\varepsilon \in (0,\tfrac{\pi}{2})$ and
  $\tau = \tfrac{2}{\pi-2\varepsilon}$, then the functions
  $b_n(t) =
  \tfrac{\sqrt{\pi}}{\sqrt{2}\ell(t)}e^{-i\pi^2n^2\tfrac{t}{\ell(t)}}$
  for $n\ge 1$ form an orthonormal system in $L_2(0,\tau)$.
\end{lemma}

\begin{proof}
  Note that
  $\bigl(\tfrac{t}{\ell(t)}\bigr)' =
  \tfrac{\ell(t)-t\ell'(t)}{\ell(t)^2} = \tfrac{1}{\ell(t)^2}$.
  Therefore, the obvious change of variable $x=\tfrac{t}{\ell(t)}$ reduces $f_n$ to a standard trigonometric function
  on $[0, \tfrac{\tau}{\ell(\tau)}]$. Observe that
  $\tfrac{\tau}{\ell(\tau)} =   \tfrac{2}{\pi-2\varepsilon}(1+\tfrac{2\varepsilon}{\pi-2\varepsilon})^{-1}
  = \tfrac{2}{\pi}$.  Now orthonormality easily follows.
\end{proof}
Observe that the above sequence $\{b_n(t)\}_{n \geq 1}$ is not an
orthonormal basis. Indeed, with
$f(t) = \frac{\sqrt{\pi}}{\sqrt{2}\ell(t)}e^{3i\pi^2\frac{t}{\ell(t)}}$,
we have $\langle f(t),b_n(t)\rangle = 0$ for all $n \in \mathbb{N}$.

\subsubsection{Neumann observation at the Boundary}

\begin{proof}[Proof of Theorem~\ref{thm:new-obs-boundary}]
  We start considering only the first term at $x=0$.  As in the proof
  of Lemma~\ref{initial-norm} we consider for a moment only initial
  data associated with finitely supported sequences $(a_n)$.
  Differentiating the representation formula
  (\ref{represetiation-formula}) $u$ term by term yields
\[
    u_x(0,t) = \sum_{n=1}^{+\infty} a_n  \bigl(\tfrac{2}{\ell(t)}\bigr)^{\onehalf}e^{-i\pi^2n^2\tfrac{t}{\ell(t)}}\tfrac{n\pi}{\ell(t)},
\]
and therefore
\[
   \norm{ u_x(0,t) }_{L_2(0,\tau)}^2
=  \;  \int_{0}^{\tau}\frac{2\pi^2}{\ell(t)^3}\left|\sum_{n=1}^{+\infty}n a_n  e^{-i\pi^2n^2\frac{t}{\ell(t)}}\right|^2 \dt.
\]
Using the monotonicity of $\ell(t)$ in $[0,\tau]$, we have $\frac{2\pi^2}{\ell(\tau)} J \; \le  \; \norm{ u_x(0,\cdot) }_{L_2(0,\tau)}^2 \; \le  \;  2\pi^2 J$
where 
\[
  J=\int_{0}^{\tau}\left|\sum_{n=1}^{+\infty}n a_n  e^{-i\pi^2n^2\frac{t}{\ell(t)}}\right|^2 \tfrac{\dt}{\ell(t)^2}.
\]
This allows to focus only on the integral $J$, where we abbreviate
$b_n = n a_n e^{-i\pi^2n^2/\varepsilon}$
and make a change of variable $\xi = \frac{-1}{\ell(t)} + \tfrac12(\frac{1}{\ell(0)}{+}\frac{1}{\ell(\tau)})$. 
Letting  $T=\frac{1}{\ell(0)} - \frac{1}{\ell(\tau)}$,
the above double inequality rewrites as 
\[
 \int_{-\nicefrac{T}{2}}^{+\nicefrac{T}{2}}   \left|\sum_{n=1}^{+\infty} b_n  e^{-i  \frac{\pi^2n^2}{\varepsilon}\,\xi}\right|^2 \dxi
\quad \approx \quad \norm{ u_x(0,t) }_{L_2(0,\tau)}^2 
\]
The sequence $\lambda_n = \frac{\pi^2n^2}{\varepsilon}$ satisfies the hypotheses of \cite[Theorem 3.1 and Corollary 3.3]{tenenbaum}
so that, for all $k > \tfrac{3}{2} \pi^2$ and $r=\nicefrac{\varepsilon}{\pi^2}$
\[
 \int_{-\nicefrac{T}{2}}^{+\nicefrac{T}{2}}   \left|\sum_{n=1}^{+\infty} b_n  e^{-i  \frac{\pi^2n^2}{\varepsilon}\,\xi}\right|^2 \dxi
\gg  \;  e^{-\frac{2k}{r \tau}} \sum_{n=1}^{+\infty}\left| b_n \right|^2  \\
=  \;  {e^{-\frac{2k}{r\tau}}}\sum_{n=1}^{+\infty}\left| n a_n  \right|^2. 
\]
On the other hand side, if
$ T  \in [m \tfrac{\varepsilon}\pi, (m{+}1) \tfrac{\varepsilon}\pi)$,
we have by periodicity and Parseval's identity
\[
\int_{-\nicefrac{T}{2}}^{+\nicefrac{T}{2}}   \left|\sum_{n=1}^{+\infty} b_n  e^{-i n^2  \frac{\pi^2}{\varepsilon}\,\xi}\right|^2 \dxi
\le
\int_{-(m{+}1)\frac{\varepsilon}{\pi}}^{(m{+}1)\frac{\varepsilon}{\pi}}   \left|\sum_{n=1}^{+\infty} b_n  e^{-i n^2 \frac{\pi^2}{\varepsilon}\,\xi}\right|^2 \dxi
=  (m{+}1)  \sum_{n=1}^{+\infty}\left| b_n \right|^2.
\]
We conclude by Lemma~\ref{initial-norm} that
\[
c(\varepsilon) \norm{ u_0  }_{H_0^1(0,1)}^2
\; \le \; \norm{ u_x(0,t) }_{L_2(0,\tau)}^2  \; \le \;
C(\varepsilon) \norm{ u_0 }_{H_0^1(0,1)}^2.
\]
This inequality being true for all $u_0$ leading to finitely supported
sequences $(a_n)$, it is true for any $u_0 \in H_0^1(0,1)$
by density.

\medskip

For second term at $x=\ell(t)$, we see for finitely supported
sequences $(a_n)$ that
\[
   u_x(\ell(t),t) = \sum_{n=1}^{+\infty} (-1)^n a_n  \bigl(\tfrac{2}{\ell(t)}\bigr)^{\onehalf}e^{-i\pi^2n^2\tfrac{t}{\ell(t)}}\tfrac{n\pi}{\ell(t)}e^{i\tfrac{\varepsilon}{4}\ell(t)}
\]
Taking the $L_2$-norm, one get the equivalent between $\norm{u_x(\ell(t),t)}_{L_2}$ and $\norm{u_x(0,t)}_{L_2}$
\begin{align*}
\norm{ u_x(\ell(t),t) }_{L_2(0,\tau)}^2
   = & \;  \int_{0}^{\tau}\left|\sum_{n=1}^{+\infty} 
         (-1)^n a_n \bigl(\tfrac{2}{\ell(t)}\bigr)^{\onehalf}e^{-i\pi^2n^2\tfrac{t}{\ell(t)}}
         \tfrac{n\pi}{\ell(t)}e^{i\tfrac{\varepsilon}{4}\ell(t)}\right|^2 \dt \\
   = & \;  \int_{0}^{\tau} \tfrac{2\pi^2}{\ell(t)} \left|\sum_{n=1}^{+\infty} 
          \; \bigl((-1)^n  n a_n\bigr) 
          e^{-i\pi^2n^2\tfrac{t}{\ell(t)}} \right|^2 \tfrac{\dt }{\ell(t)^2}.
\end{align*}
Clearly, the rest proof follows the lines above.
\end{proof}

\subsubsection{Internal Point Observability}
\begin{proof}[Proof of Theorem~\ref{thm:observation-at-internal-points}]
Since $\ell(t) \geq 1$ for all $t$, 
\[
\int_{0}^{\tau}\frac{2}{\ell(t)}\left|\sum_{n=1}^{+\infty}a_n e^{-i\pi^2n^2\tfrac{t}{\ell(t)}}\sin\bigl(\tfrac{n\pi a}{\ell(t)}\bigr)\right|^2 \dt 
\geq 
  \int_{0}^{\tau}\frac{2}{\ell(t)^2}\left|\sum_{n=1}^{+\infty}a_n e^{-i\pi^2n^2\tfrac{t}{\ell(t)}}\sin\bigl(\tfrac{n\pi a}{\ell(t)}\bigr)\right|^2 \dt.
\]
By definition,
$\sin\bigl(\tfrac{n\pi a}{\ell(t)}\bigr) = \frac{1}{2i}\bigl(\exp({i
  \tfrac{n\pi a}{\ell(t)}})-\exp({-i \tfrac{ n \pi a}{\ell(t)}})
\bigr)$. 
Therefore,
\begin{align*}
\sum_{n=1}^{+\infty}a_n e^{-i\pi^2n^2\tfrac{t}{\ell(t)}}\sin\bigl(\tfrac{n\pi a}{\ell(t)}\bigr)
 =  & \; \frac{1}{2i}\sum_{n=1}^{+\infty}a_n e^{-i\pi^2n^2\tfrac{t}{\ell(t)}}\bigl(e^{\frac{i n\pi a}{\ell(t)}}-e^{-\frac{i n \pi a}{\ell(t)}}\bigr) \\
=  & \;  \frac{1}{2i}\sum_{n=1}^{+\infty}a_n e^{-i\pi^2n^2\tfrac{1}{\varepsilon}}\bigl(e^{-\frac{i\pi^2n^2}{\varepsilon \ell(t)}+ \frac{i n\pi a}{\ell(t)}}-e^{-\frac{i\pi^2n^2}{\varepsilon \ell(t)}-\frac{i n \pi a}{\ell(t)}}\bigr)
\end{align*}
For $n \in \ZZ$, we extend the series by $a_{n} = a_{-n}$, and
$\lambda_{n} = \frac{\pi^2n^2}{\varepsilon} + \text{sign}(n) n\pi a$.
The sequence $\lambda_n = \frac{\pi^2n^2}{\varepsilon}$ is regular and
satisfies the hypotheses of \cite[Theorem 3.1]{tenenbaum} with
$r = \frac{\varepsilon}{\pi^2}$ and $C = a\pi$. We follow the lines of
the proof of Theorem~\ref{thm:new-obs-boundary}: changing the
variable $\xi =\frac{-1}{\ell(t)}$ gives with the notation
$T=\frac{1}{\ell(0)} - \frac{1}{\ell(\tau)}$,
\[
\int_{0}^{\tau}\frac{1}{\ell(t)^2}\left|\sum_{n=1}^{+\infty}a_n e^{-i\pi^2n^2\tfrac{t}{\ell(t)}}\sin\bigl(\tfrac{n\pi a}{\ell(t)}\bigr)\right|^2 \dt  
= 
\frac{1}{\varepsilon}\int_{-\nicefrac{T}{2}}^{+\nicefrac{T}{2}} \left|\sum_{n \in \ZZ}^{}  e^{\frac{-i\pi^2n^2}{\varepsilon}} a_n   e^{i\lambda_n \xi}\right|^2 \dxi
\]
we write $b_n = e^{\frac{-i\pi^2n^2}{\varepsilon}} a_n$ and use
\cite[Corollary 3.3]{tenenbaum} with $k > \frac{3\pi^2}{2}$:
\[
 \frac{1}{\varepsilon}\int_{-T}^{T}\left|\sum_{n \in \ZZ}^{}a_n e^{\frac{-i\pi^2n^2}{\varepsilon}}e^{-i\lambda_n \xi}\right|^2 \dxi 
\; \gg \;  e^{-\frac{2k}{rT}}\sum_{n \in \ZZ}^{}\left| a_n e^{\frac{-i\pi^2n^2}{\varepsilon}}\right|^2 
\ge  {e^{-\frac{2k}{rT}}}\sum_{n =1}^{+\infty}\left| a_n \right|^2. \qedhere
\]
For the upper estimate, we use similar method as in theorem (\ref{thm:new-obs-boundary}). More precisely,
\[
\norm{u(a,t)}_{L_2} \leq \int_{0}^{\tau}\frac{2}{\ell(t)}\left|\sum_{n=1}^{+\infty}a_n e^{-i\pi^2n^2\tfrac{t}{\ell(t)}}\right|^2 \dt \lesssim (m+1)\sum_{n=1}^{+\infty}|a_n|^2
\]
where $m$ be the integer number such that $\frac{\pi\varepsilon}{T} \in [m,m+1]$ with $T = \frac{1}{\ell(0)}-\frac{1}{\ell(\tau)}$.
\end{proof}

\subsubsection{$L_p$-admissibility and observability}
\begin{proof}[Proof of Theorem~\ref{thm:Lp-ell2-observation}]
  The upper estimate yielding $K_p(\tau)$ is obtained by interpolation
  of the two upper estimates in Theorem~\ref{thm:new-obs-boundary}.
  We are left with the lower estimate.  Since $u \in H_0^1$,
  $(n a_n) \in \ell_2$, and so $(a_n) \in \ell_1$ by the
  Cauchy-Schwarz inequality.  Let $p \in (0,2)$ and let
  $\theta = \frac{2}{4-p} \in (0,1)$ which is chosen to satisfy
  $p\theta + 4(1-\theta) = 2$. By Hölder's inequality we then have
\begin{equation}\label{Holder-estimate}
  \begin{aligned}
        & \int_{0}^{\tau}\bigl|u(a,t)\bigr|^2 \,dt 
        = \int_{0}^{\tau}\bigl|u(a,t)\bigr|^{p\theta} . \bigl|u(a,t)\bigr|^{4(1-\theta)} \,dt \\
\leq \; & \Bigl(\int_{0}^{\tau}\bigl|u(a,t)\bigr|^p dt\Bigr)^{\theta} . 
          \Bigl(\int_{0}^{\tau}\bigl|u(a,t)\bigr|^4 dt \Bigr)^{1-\theta}
  \end{aligned}
\end{equation}
From trivial argument on boundedness of $\sin(\tfrac{n\pi a}{\ell(t)})$ and 
$e^{\tfrac{i\varepsilon a^2}{4\ell(t)}-i\pi^2n^2\tfrac{t}{\ell(t)}}$:
\[
  \bigl|u(a,t)\bigr|^2 
= \Bigl|\sum_{n=1}^{+\infty}a_n e^{\frac{i\varepsilon a^2}{4\ell(t)}-i\pi^2n^2\tfrac{t}{\ell(t)}}\sin\bigl(\tfrac{n\pi a}{\ell(t)}\bigr)\Bigr|^2 
\leq \Bigl(\sum_{n=1}^{+\infty}|a_n| \Bigr)^2
\]
Combining with the estimate (\ref{Holder-estimate}), one get:
\begin{equation}  \label{eq:4-to-2}
\int_{0}^{\tau}\bigl|u(a,t)\bigr|^4 dt 
\leq \Bigl(\sum_{n=1}^{+\infty}|a_n|\Bigr)^2 \Bigl(\int_{0}^{\tau}\bigl|u(a,t)\bigr|^2 \,dt \Bigr) 
\end{equation}
From inequalities (\ref{Holder-estimate}) and (\ref{eq:4-to-2}) and Theorem~(\ref{thm:observation-at-internal-points}) we deduce now
\begin{align*}
\int_{0}^{\tau}\bigl|u(a,t)\bigr|^p dt 
\geq & \; \Bigl( \int_0^\tau |u(a, t)|^2\,dt \Bigr)^{\nicefrac{1}{\theta}}  \Bigl( \int_0^\tau |u(a, t)|^4\,dt \Bigr)^{\frac{\theta-1}{\theta}}\\
\geq  & \; \Bigl( \int_0^\tau |u(a, t)|^2\,dt \Bigr)^{\nicefrac{1}{\theta}} 
  \Bigl(\sum_{n=1}^{+\infty}|a_n| \Bigr)^{\frac{2(\theta -1)}{\theta}}\Bigl(\int_{0}^{\tau}\bigl|u(a,t)\bigr|^2 dt \Bigr)^{\frac{\theta-1}{\theta}} \\
\geq & \; k \Bigl(\sum_{n=1}^{+\infty}|a_n|^2 \Bigr)
              \Bigl(\sum_{n=1}^{+\infty}|n a_n|^2\Bigr)^{2\frac{\theta-1}{\theta}} \geq k \bignorm{u_0}_{L_2(0,1)}^2 \bignorm{ u_0 }_{H_0^1}^{2\frac{\theta-1}{\theta}}.
\end{align*}
Since $\frac{\theta-1}{\theta} = \frac{p-2}2$, the result follows.
\end{proof}

\section{Boundary controllability of dual problem}

Since we have already stated several theorems that can be interpreted as
exact observation we will briefly sketch the duality theory that
allows to rephrase these assertions in terms of exact control, then the solution $z$ to adjoint problem
\begin{equation}  \label{eq:retrograde}
    z'(t) = -A(t)^* z(t) - C(t)^*C(t) w(t) \qquad z(\tau)=0
\end{equation}
satisfies
$\langle w_0,z(0) \rangle = -\int_0^\tau \tfrac{d}{dt} \langle
  w(t), z(t)\rangle  \,dt = \int_0^\tau \|C(t)w(t)\|^2 \,dt$
by injection of the respective differential equations of $w$ and $z$.
Hence exact observability implies that the Gramian
$Q: w_0 \mapsto z(0)$ satisfies
$\|Qw_0\| \|w_0\| \ge \langle w_0 , Q w_0 \rangle \ge \delta \|w_0\|$
to the effect that $Q$ has closed image. Moreover, if $Q^* w_0 = 0$,
taking scalar product with $w_0$ reveals $w_0=0$, so $Q^*$ is
injective and hence $Q$ has dense range. By the open mapping theorem,
$Q$ is therefore an isomorphism on $X$. This means that the adjoint
problem (\ref{eq:retrograde}) can be steered to any state
$z(0) \in X$ by an appropriate choice of the initial value $w_0$. Indeed, for $u, v \in D(A(t))$
we have
\begin{align*}
\langle A(t)u, v \rangle_X  
= & \;   \Bigl\langle \tfrac{i}{\ell(t)^2}u_{yy}+\tfrac{\ell'(t)}{\ell(t)}yu_y , v \Bigr\rangle_X  =  \int_0^{1} \tfrac{i}{\ell(t)^2}u_{yy}\overline{v}\,dy + \int_{0}^{1}\tfrac{\ell'(t)}{\ell(t)}yu_y\overline{v}\,dy  \\
\text{(int. by parts)} & \; = -\tfrac{i}{\ell(t)^2}\int_0^{1} u_y \overline{v}_{y}\, dy - \tfrac{\ell'(t)}{\ell(t)}\int_{0}^{1}(yu\overline{v}_y+u\overline{v})\,dy\\
\text{(int. by parts)} & \; = \tfrac{i}{\ell(t)^2}\int_0^{1}u\overline{v}_{yy}\,dy - \tfrac{\ell'(t)}{\ell(t)}\int_{0}^{1}(yu\overline{v}_y+u\overline{v})\,dy\\
= & \; -\Bigl\langle u,  \tfrac{i}{\ell(t)^2}v_{yy}+\tfrac{\ell'(t)}{\ell(t)}yv_y \Bigr\rangle - \Bigl\langle u, \tfrac{\ell'(t)}{\ell(t)}v\Bigr\rangle = \Bigl\langle u, -\bigl(A(t)+\tfrac{\ell'(t)}{\ell(t)}\bigr)v\Bigr\rangle
\end{align*}
It turns out that in our case $A(t)^* = -A(t)-\frac{\ell'(t)}{\ell(t)}$. So exact observation of the Schrödinger equation
(\ref{initial-system-P}) can be reformulated as exact control
for the Schrödinger equation with zero final time.  We turn back to these
ideas after stating our first theorem. In the case of linear moving $\ell(t) = 1{+}\varepsilon t$, let
$C(t) : D(A(t)) \to \mathbb{C}$ be given by
$C(t) (\varphi) := \varphi_y(b)$ where $b \in \{0,1\}$. The (lower) estimate in theorems~
\ref{thm:new-obs-boundary} and \ref{thm:obs-inequality} then reformulates as exact
observability of $C(t)$ for the non-autonomous Cauchy problem
(\ref{eq:abstract-nonaut-pb}). Some care has to be taken since $C(t)$
is unbounded on $X$. Indeed, $C(t)^{*}: \mathbb{C} \to D(A(t))'$ is
given by
$C(t)^* \alpha = - \alpha \, \tfrac{d}{dy} \delta_{y=b}$, then
we obtain exact controllability of (\ref{eq:retrograde}) in a
distributional sense:
\[
 z_t = \tfrac{i}{\ell(t)^2}z_{yy} + \tfrac{\ell'(t)}{\ell(t)}yz_y+\tfrac{\ell'(t)}{\ell(t)}z + w_y(b,t)\tfrac{d}{dy} \delta_{y=b}
\qquad \text{and}\quad  z(y,\tau) = 0
\]
Multiplying with a test function
$\eta \in D((0, 1))$, and integrating on $[0, 1]$ we
obtain by partial integration
\begin{align*}
   \int_0^{1} z_{t}\eta(y)\, dy \; &  = \; \int_0^{1}\Bigl(\tfrac{i}{\ell(t)^2}z_{yy}+\tfrac{\ell'(t)}{\ell(t)}(yz)_y\Bigr) \eta(y)\,dy  - w_y(b, t) \eta'(b)  \\
= & \; \int_0^{1}\Bigl(\tfrac{i}{\ell(t)^2}z\eta''(y)-\tfrac{\ell'(t)}{\ell(t)}yz\eta'(y)\Bigr)\,dy + \Bigl(\tfrac{i}{\ell(t)^2}z(b, t) - w_y(b , t)\Bigr) \, \eta'(b)  
\end{align*}
This is possible for any test function $\eta$ only if the point
evaluation vanishes. The dual statement of the lower estimate in theorems~\ref{thm:new-obs-boundary} and \ref{thm:obs-inequality} is thus exact controllability of a
Schrödinger  equation with Dirichlet control on the right boundary,
\begin{equation}
  \label{eq:dual-problem}
  \left\{
    \begin{array}{rll}
      z_{t} & = \tfrac{i}{\ell(t)^2}z_{yy}+\tfrac{\ell'(t)}{\ell(t)}yz_y+\tfrac{\ell'(t)}{\ell(t)}z  \qquad                        & (y, t) \in (0,1)\times (0,\tau)\\
      z(\overline{b}, t) & = 0  \qquad     & \{\overline{b}\}\bigcup\{b\} = \{0,1\}, t \ge 0\\
      z(b, t) & = -i\ell(t)^2w_y(b, t)\qquad     & t \ge 0\\
      z(y,\tau) & = 0   \qquad       & y \in [0,1]\\
    \end{array}
  \right.
\end{equation}
We reverse back to the moving boundary problem by taking $x = \ell(t)y$ and $h(x,t) = z(y,t)$. Then the problem can be  written as:
\begin{equation}
  \label{eq:dual-problem2}
  \left\{
    \begin{array}{rll}
      ih_{t}+h_{xx}-i\tfrac{\ell'(t)}{\ell(t)}h & = 0  \qquad  & (x, t) \in (0,\ell(t))\times (0,\tau)\\
      h(\ell(t), t) & = 0  \qquad     &  t \ge 0\\
      h(0, t) & = -i\ell(t)^3u_x(0, t)\qquad     & t \ge 0\\
      h(x,\tau) & = 0   \qquad       & x \in [0,\ell(t)]\\
    \end{array}
  \right.
\end{equation}
or
\begin{equation}
  \label{eq:dual-problem3}
  \left\{
    \begin{array}{rll}
      ih_{t}+h_{xx}-i\tfrac{\ell'(t)}{\ell(t)}h & = 0  \qquad  & (x, t) \in (0,\ell(t))\times (0,\tau)\\
      h(0, t) & = 0  \qquad     &  t \ge 0\\
      h(\ell(t), t) & = -i\ell(t)^3u_x(\ell(t), t)\qquad     & t \ge 0\\
      h(x,\tau) & = 0   \qquad       & x \in [0,\ell(t)]\\
    \end{array}
  \right.
\end{equation}
In general situation of $\ell(t)$ satisfying condition $(\ref{condition on the curve})$, one take
$C(t) : D(A(t)) \to \mathbb{C}\times \mathbb{C}$ be given by
$C(t) (\varphi) := (\varphi_y(0),\varphi_y(1))$. Therefore, the dual operator $C(t)^{*}: \mathbb{C}\times \mathbb{C} \to D(A(t))'$ is
given by $C(t)^* (\alpha,\beta) = - \alpha \, \tfrac{d}{dy} \delta_{y=0}-\beta\tfrac{d}{dy}\delta_{y=1}$. Using similar argument, we obtain exact controllability of a
Schrödinger equation with Dirichlet control applied on both of boundaries
\begin{equation}
  \label{eq:dual-problem4}
  \left\{
    \begin{array}{rll}
      ih_{t}+h_{xx}-i\tfrac{\ell'(t)}{\ell(t)}h & = 0  \qquad  & (x, t) \in (0,\ell(t))\times (0,\tau)\\
      h(0, t) & = -i\ell(t)^3u_x(0, t)  \qquad     &  t \ge 0\\
      h(\ell(t), t) & = -i\ell(t)^3u_x(\ell(t), t)\qquad     & t \ge 0\\
      h(x,\tau) & = 0   \qquad       & x \in [0,\ell(t)]\\
    \end{array}
  \right.
\end{equation}

\subsection*{Acknowledgements}
This work was undertaken as part of the authors PhD thesis at the University of Bordeaux. The author wants to thank his advisers Bernhard Haak and El-Maati Ouhabaz for their patiently supports, great motivations, and immense knowledges.\\ 

\noindent Further, we kindly acknowledge valuable discussions with  Marius Tucsnak.


\end{document}